\documentclass[oneside,english]{amsart}

\usepackage[T1]{fontenc}
\usepackage[utf8]{inputenc}
\usepackage{amsmath,amstext,amsthm,amssymb}
\usepackage{babel}

\theoremstyle{plain}
\newtheorem{theorem}{Theorem}[section]
\newtheorem{lemma}[theorem]{Lemma}
\newtheorem{proposition}[theorem]{Proposition}

\begin{document}

\title[Reduced-Trace-Zero Noncommutator]{A Reduced-Trace-Zero Element That Is Not a Commutator in a Central Division Algebra}

\author{Hau-Yuan Jang}
\address{Department of Mathematics, National Cheng Kung University, No.~1, Dasyue Road, Tainan 70101, Taiwan}
\email{L18121022@gs.ncku.edu.tw}

\subjclass[2020]{Primary 16K20; Secondary 16W60}
\keywords{central simple algebra, division algebra, commutator, reduced trace, skew Laurent series}

\begin{abstract}
We exhibit a finite-dimensional central division algebra, of degree
four over its center, together with an explicit element of reduced
trace zero that is not a single commutator. This answers negatively
the single-commutator question for finite-dimensional central division
algebras, the case left open by Amitsur and Rowen. The algebra is
a skew Laurent series ring, the element has only two terms, and the
proof uses the $x$-adic valuation, a first-obstruction lemma based
on two commuting involutions of the associated graded ring, and a
parity argument for quadratic forms over iterated Laurent series fields.
\end{abstract}

\maketitle

\section{Introduction}

A standard exercise in linear algebra states that a square matrix
over a field is a commutator exactly when its trace vanishes: if $A\in M_{n}(k)$
and $\operatorname{Tr}(A)=0$, then $A=XY-YX$ for suitable $X,Y$. The result
is due to Shoda and Albert--Muckenhoupt \cite{Shoda,AlbertMuckenhoupt}.
The analogous question for a central simple algebra $R$ is more delicate.
Every commutator in $R$ has reduced trace zero, and one asks whether
the converse holds: is every element of reduced trace zero a single
commutator?

Amitsur and Rowen \cite{AmitsurRowen} studied this systematically.
They showed that every element of reduced trace zero is a sum of at
most two commutators, and that in the non-division case every non-scalar
such element is already a single commutator. For division algebras
they left the question open, observing that it seemed to be governed
by quadratic forms.

The purpose of this note is to settle that case in the negative. We
construct a cyclic skew Laurent series division algebra 
\[
D=F((x;\sigma)),
\]
of degree $4$ over its center $Z(D)=K((x^{4}))$, and an explicit
element $A\in D$ of reduced trace zero that is not a single commutator.

\subsection*{Strategy of the proof}

The construction and the proof use four standard ingredients.  We work
in a skew Laurent series division algebra and use the trace pairing for
finite separable field extensions to understand the coefficient maps
arising from homogeneous commutators.  The $x$-adic filtration and its
associated graded algebra are then used to reduce a hypothetical
identity $A=[B,C]$ to a minimal leading-term configuration.  Finally,
the surviving coefficient equations give a representation problem for
quadratic forms over iterated Laurent series fields.

More precisely, the filtration argument allows us to choose a
commutator representation satisfying
\[
 v_x(B)+v_x(C)=v_x(A).
\]
The possible homogeneous leading terms can then be analysed in the skew
Laurent polynomial ring associated to the filtration.  After the odd
leading configurations have been excluded, an explicit calculation of
the next coefficients forces an identity of the form
\[
 rz=-\frac{2}{z(\alpha^2+\beta^2s)}
 \left(\operatorname{N}_{L/K}(P)+\iota s\,\operatorname{N}_{L/K}(Q)\right).
\]
The expression in parentheses is represented by the quadratic form
\[
 \langle 1,-s,\iota s,-\iota s^2\rangle,
\]
while the denominator contains a value of $\langle 1,s\rangle$.
Springer's theorem gives the required anisotropy and hence an even
$r$-adic valuation on the right-hand side, whereas $v_r(rz)=1$.
This parity contradiction proves the theorem.

For background on skew Laurent series division rings and their
valuations, see \cite[Example~1.8 and \S14]{LamNCR}.  For norm, trace,
and the nondegenerate trace pairing of a finite separable extension, see
\cite[Chapter~VI, \S5, especially Theorem~5.2]{Lang}; for filtered and
associated graded algebras and modules, see
\cite[Chapter~X, \S5]{Lang}.  The quadratic-form input is the form of
Springer's theorem recalled in
\cite[Chapter~VI, \S1, Proposition~1.9]{Lam}.

\section{Some observations}
\label{sec:observations}

We first recall the elementary facts about skew Laurent series that
will be used throughout, and then record several observations that
motivated the construction of the counterexample.  The latter are
included to explain how the example was found and are not needed in
the subsequent proof.

Let $F/K$ be a cyclic Galois extension with
\[
\operatorname{Gal}(F/K)=\langle \sigma\rangle
\]
of order $n$.  The skew Laurent series division ring
\[
D=F((x;\sigma))
\]
consists of formal Laurent series
\[
\sum_{m\ge m_0}a_mx^m,\qquad a_m\in F,
\]
with multiplication determined by
\[
xa=\sigma(a)x\qquad(a\in F).
\]
Thus
\[
x^ia=\sigma^i(a)x^i,
\qquad
(ax^i)(bx^j)=a\sigma^i(b)x^{i+j}.
\]
For background on skew Laurent series division rings, see
\cite[Example~1.8 and \S14]{LamNCR}.

Since $\sigma$ has order $n$, one has
\[
Z(D)=K((x^n)).
\]
Moreover, $[D:Z(D)]=n^2$, so $D$ is a central division algebra
of degree $n$.

The least exponent defines the $x$-adic valuation.  For
\[
0\ne d=\sum_{m\ge m_0}a_mx^m,\qquad a_{m_0}\ne0,
\]
put
\[
v_x(d)=m_0.
\]
Then
\[
v_x(dd')=v_x(d)+v_x(d')
\]
and
\begin{equation}
v_x([d,d'])\ge v_x(d)+v_x(d').
\label{eq:comm-val}
\end{equation}

We shall also use the reduced trace explicitly.  Put
\[
t=x^n,\qquad Z=K((t)),\qquad E=F((t)).
\]
Then $Z=Z(D)$, $E/Z$ is a cyclic extension of degree $n$, and
\[
D=\bigoplus_{j=0}^{n-1}x^jE
\]
as a right $E$-vector space.  The left regular representation on
this space gives the standard splitting representation
\[
D\otimes_ZE\simeq M_n(E),
\]
so the reduced trace of an element of $D$ is the ordinary matrix
trace in this representation.

Write
\[
A=\sum_m a_mx^m
  =\sum_{r=0}^{n-1} e_rx^r,
\qquad
e_r=\sum_q a_{qn+r}t^q\in E.
\]
With respect to the right $E$-basis
\[
1,x,\ldots,x^{n-1},
\]
left multiplication by $e_rx^r$ sends the $j$th basis line
$x^jE$ to $x^{j+r}E$, where the exponent is read modulo $n$.
Hence, if $r\ne0$, the corresponding matrix has zero diagonal and
therefore
\[
\operatorname{Trd}_{D/Z}(e_rx^r)=0.
\]

For $r=0$, left multiplication by $e_0$ is diagonal: since
\[
e_0x^j=x^j\sigma^{-j}(e_0),
\]
its diagonal entries are
\[
e_0,\sigma^{-1}(e_0),\ldots,\sigma^{-(n-1)}(e_0).
\]
Consequently
\[
\operatorname{Trd}_{D/Z}(e_0)
=
\operatorname{Tr}_{E/Z}(e_0).
\]
As $\sigma$ acts coefficientwise on $E=F((t))$, this gives
\begin{equation}
\operatorname{Trd}_{D/Z(D)}(A)
=
\sum_q
\operatorname{Tr}_{F/K}(a_{qn})t^q
=
\sum_{m\equiv0\pmod n}
\operatorname{Tr}_{F/K}(a_m)x^m.
\label{eq:trd}
\end{equation}

We now turn to the coefficient calculations that motivated the
example.  For homogeneous terms,
\[
[ax^i,bx^j]
=
\bigl(a\sigma^i(b)-\sigma^j(a)b\bigr)x^{i+j}.
\]
It is therefore natural, for $a\in F$ and $i,j\in\mathbb Z$, to
consider the $K$-linear map
\[
f_a^{i,j}:F\longrightarrow F,
\qquad
f_a^{i,j}(b)
=
a\sigma^i(b)-\sigma^j(a)b.
\]
Then
\[
[ax^i,bx^j]=f_a^{i,j}(b)x^{i+j}.
\]
Thus the possible coefficients of a homogeneous commutator with
fixed $a,i,j$ are precisely the elements of
$\operatorname{Im}f_a^{i,j}$.

The following elementary observation describes these images through
the trace pairing.  We use the standard nondegeneracy of
\[
(u,v)\longmapsto\operatorname{Tr}_{F/K}(uv)
\]
for finite separable extensions; see
\cite[Chapter~VI, \S5, Theorem~5.2]{Lang}.

\begin{lemma}
\label{lem:image-observation}
Let $a\in F^\times$ and $i,j\in\mathbb Z/n\mathbb Z$.  Put
\[
E_a^{i,j}
=
\{\lambda\in F:
  \sigma^i(\lambda)\sigma^{i+j}(a)=a\lambda\}.
\]
Then
\[
\operatorname{Im}f_a^{i,j}
=
(E_a^{i,j})^\perp
=
\left\{
u\in F:
\operatorname{Tr}_{F/K}(\lambda u)=0
\text{ for every }\lambda\in E_a^{i,j}
\right\}.
\]
\end{lemma}

\begin{proof}
Let $F^*=\operatorname{Hom}_K(F,K)$.  By the nondegeneracy of the
trace pairing, every element of $F^*$ is uniquely of the form
\[
u\longmapsto\operatorname{Tr}_{F/K}(\lambda u)
\]
for some $\lambda\in F$.  For $b\in F$,
\begin{align*}
\operatorname{Tr}_{F/K}
 \bigl(\lambda f_a^{i,j}(b)\bigr)
&=
\operatorname{Tr}_{F/K}
 \bigl(\lambda a\sigma^i(b)
       -\lambda\sigma^j(a)b\bigr)\\
&=
\operatorname{Tr}_{F/K}
 \bigl(
   (\sigma^{-i}(\lambda a)-\lambda\sigma^j(a))b
 \bigr).
\end{align*}
This vanishes for every $b\in F$ if and only if
\[
\sigma^{-i}(\lambda a)=\lambda\sigma^j(a),
\]
or equivalently
\[
\sigma^i(\lambda)\sigma^{i+j}(a)=a\lambda.
\]
Hence, under the trace-pairing identification $F\simeq F^*$,
the annihilator of $\operatorname{Im}f_a^{i,j}$ is precisely
$E_a^{i,j}$. Since the trace pairing is nondegenerate and
$\operatorname{Im}f_a^{i,j}$ is a $K$-subspace of $F$, taking
orthogonal complements once more gives
\[
\operatorname{Im}f_a^{i,j}
=
\bigl((\operatorname{Im}f_a^{i,j})^\perp\bigr)^\perp
=
(E_a^{i,j})^\perp.
\]
This proves the lemma.
\end{proof}

We now specialize to $n=4$.  Put
\[
L=F^{\sigma^2},
\qquad
\tau=\sigma^2.
\]
Thus $F/L$ is quadratic.  Lemma~\ref{lem:image-observation} gives
\begin{align}
\operatorname{Im}f_a^{2,2}
&=\ker\operatorname{Tr}_{F/L},
\label{eq:obs-22}\\
\operatorname{Im}f_a^{1,3}
&=\ker\operatorname{Tr}_{F/K},
\label{eq:obs-13}\\
\operatorname{Im}f_a^{2,0}
&=
\left\{
u\in F:
\operatorname{Tr}_{F/L}(a^{-1}u)=0
\right\},
\label{eq:obs-20}\\
\operatorname{Im}f_a^{1,1}
&=
\left\{
u\in F:
\operatorname{Tr}_{F/K}
 \left(\frac{u}{a\sigma(a)}\right)=0
\right\}.
\label{eq:obs-11}
\end{align}

The last two formulas show explicitly that, once a leading
coefficient has been fixed, the next coefficient may be constrained
to a trace hyperplane depending on that leading coefficient.

There is also a useful norm feature in the $(2,2)$ case.  Writing
$b=a\theta, \theta\in F$, one obtains
\begin{equation}
f_a^{2,2}(b)
=
a\tau(a)(\tau(\theta)-\theta)
=
\operatorname{N}_{F/L}(a)(\tau(\theta)-\theta).
\label{eq:obs-norm-22}
\end{equation}
Since $\operatorname{N}_{F/L}(a)\in L^\times$, this does not change
the full image:
\[
\operatorname{Im}f_a^{2,2}
=
\ker\operatorname{Tr}_{F/L}.
\]
In successive coefficient equations, an earlier equation may express
one coefficient in terms of another. Substituting this relation into
the next bilinear coefficient equation naturally produces quadratic
expressions; in the $(2,2)$ case, these include the norm
$\operatorname{N}_{F/L}(a)=a\tau(a)$.

The $(1,3)$ case has a complementary feature.  Although its total
image is simply $\ker\operatorname{Tr}_{F/K}$, taking trace only to
the intermediate quadratic field gives
\begin{equation}
\operatorname{Tr}_{F/L}\bigl(f_a^{1,3}(b)\bigr)
=
f_a^{3,1}(\tau(b))
+
f_{\tau(a)}^{3,1}(b).
\label{eq:obs-partial-trace-13}
\end{equation}
Thus information involving the leading coefficient $a$ survives
after passing to the intermediate trace.

These calculations suggested looking in the residue class
$-2$ modulo $4$ and, in particular, at a sparse element of the form
\begin{equation}
ux^{-2}+v.
\label{eq:historical-ansatz}
\end{equation}
For the degree-zero coefficient it is natural to require
\[
\operatorname{Tr}_{F/K}(v)=0
\qquad\text{but}\qquad
\operatorname{Tr}_{F/L}(v)\ne0.
\]
The first condition makes the degree-zero term invisible to the
reduced trace, whereas the second preserves information over the
quadratic intermediate field. Setting the coefficient of $x^{-1}$
equal to zero makes the first few commutator equations especially
rigid. The element constructed in the next section realizes precisely
this pattern.

\section{The algebra and the element}

\label{sec:counterexample}

Let 
\[
K_{0}=\mathbb{Q}(\iota)((s)),K=K_{0}((r)),\iota=\sqrt{-1}.
\]
Choose $y$ with $y^{4}=s$. Then $X^{4}-s$ is irreducible over $K$:
by the standard binomial irreducibility criterion, since $\iota\in K$
gives $-4=(1+\iota)^{4}\in K^{4}$, it suffices to check that $s\notin K^{\times2}$.
Indeed, suppose that $q^{2}=s$ for some $q\in K^{\times}$. Since
$v_{r}(s)=0$, we have 
\[
2v_{r}(q)=v_{r}(s)=0,
\]
and hence $v_{r}(q)=0$. Thus 
\[
q=a_{0}+a_{1}r+a_{2}r^{2}+\cdots
\]
for some $a_{j}\in K_{0}$ with $a_{0}\ne0$. Comparing the constant
coefficients in the equality $q^{2}=s$ gives 
\[
a_{0}^{2}=s
\]
in $K_{0}=\mathbb{Q}(\iota)((s))$, which is impossible because every
square in $K_{0}$ has even $s$-valuation, whereas $v_{s}(s)=1$.

Since $X^{4}-s$ is irreducible and $\iota\in K$, all four roots
$y,\iota y,-y,-\iota y$ lie in $F=K(y)$. Thus $F/K$ is Galois,
with 
\[
\operatorname{Gal}(F/K)=\langle\sigma\rangle,\sigma(y)=\iota y,
\]
where $\sigma$ has order $4$. Form 
\[
D=F((x;\sigma)),L=F^{\sigma^{2}}=K(z),z=y^{2},
\]
so $z^{2}=s$ and $\sigma(z)=-z$. Throughout, a bar denotes the nontrivial
automorphism of $L/K$, $\overline{z}=-z$, and $\operatorname{N}_{L/K}(\ell)=\ell\overline{\ell}$.

The element is 
\[
A=(1+z)x^{-2}+\frac{r}{2}z\in D.
\]
Its only coefficient in a degree divisible by $4$ is $\frac{r}{2}z$,
and 
\[
\operatorname{Tr}_{F/K}\left(\frac{r}{2}z\right)=\operatorname{Tr}_{L/K}\left(\operatorname{Tr}_{F/L}\left(\frac{r}{2}z\right)\right)=\operatorname{Tr}_{L/K}(rz)=rz+\overline{rz}=0,
\]
so $A$ has reduced trace zero by \eqref{eq:trd}.
\begin{lemma}
\label{lem:iota-nonsquare}
Neither $\iota$ nor $-\iota$ is a square in $\mathbb{Q}(\iota)$.
\end{lemma}

\begin{proof}
Suppose first that
\[
(a+b\iota)^2=\iota
\]
for some $a,b\in\mathbb{Q}$.  Comparing real and imaginary parts gives
\[
a^2-b^2=0,
\qquad
2ab=1.
\]
Thus $b=\pm a$.  If $b=a$, then
\[
2a^2=1,
\]
so $a^2=\frac12$, which is impossible for $a\in\mathbb{Q}$.
Indeed, if $a=p/q$ in lowest terms, then $2p^2=q^2$, forcing both
$p$ and $q$ to be even.  If $b=-a$, then
\[
-2a^2=1,
\]
which is also impossible.  Hence $\iota\notin\mathbb{Q}(\iota)^{\times2}$.

If $-\iota=u^2$ for some $u\in\mathbb{Q}(\iota)$, then
\[
(\iota u)^2=-u^2=\iota,
\]
contradicting the first part.  Thus
$-\iota\notin\mathbb{Q}(\iota)^{\times2}$ as well.
\end{proof}

We shall also need the following elementary lemma, which rules out
the odd leading configurations in \S\ref{sec:minimal}. 
\begin{lemma}
\label{lem:odd-obstruction} For $p=1,3$ and for all $a,b\in F$,
one has 
\[
1+z\ne f_{a}^{p,p}(b).
\]
\end{lemma}

\begin{proof}
We may assume $a\ne0$, since the assertion is clear for $a=0$. Put
$\zeta=\iota^{p}$ and $t=b/a$. If $1+z=f_{a}^{p,p}(b)$, then 
\[
\frac{1+z}{a\sigma^{p}(a)}=\sigma^{p}(t)-t.
\]
The right-hand side has trace zero from $F$ to $K$. Set 
\[
F_{0}=\mathbb{Q}(\iota)((y)),K_{0}=\mathbb{Q}(\iota)((s)),s=y^{4}.
\]
Then $F=F_{0}((r))$ and $K=K_{0}((r))$, and the trace $\operatorname{Tr}_{F/K}$
is obtained coefficientwise from $\operatorname{Tr}_{F_{0}/K_{0}}$. Write
$a=r^{e}a_{0}+O(r^{e+1})$ with $a_{0}\in F_{0}^{\times}$. Taking
the leading $r$-coefficient in the trace gives 
\begin{equation}
\operatorname{Tr}_{F_{0}/K_{0}}\left(\frac{1+y^{2}}{a_{0}\sigma^{p}(a_{0})}\right)=0.\label{eq:odd-trace-zero}
\end{equation}
We will show that this is impossible.

Write $a_{0}=y^{d}u(y)$, where $d\in\mathbb{Z}$ and $u(y)\in\mathbb{Q}(\iota)[[y]]^{\times}$.
Write 
\[
u(y)=u_{0}(1+\alpha y+\beta y^{2}+O(y^{3})),u_{0}\in\mathbb{Q}(\iota)^{\times},\alpha,\beta\in\mathbb{Q}(\iota).
\]
Since $\sigma^{p}(y)=\zeta y$ and $\zeta^{2}=-1$, we have 
\[
\frac{1+y^{2}}{a_{0}\sigma^{p}(a_{0})}=\iota^{-pd}y^{-2d}(1+y^{2})\bigl(u(y)u(\zeta y)\bigr)^{-1}.
\]
The trace from $F_{0}=\mathbb{Q}(\iota)((y))$ to $K_{0}=\mathbb{Q}(\iota)((y^{4}))$
is four times the sum of the terms whose $y$-exponent is divisible
by $4$. Consequently, to rule out vanishing, it is enough to exhibit
one nonzero coefficient whose $y$-exponent is divisible by $4$.

If $d$ is even, the term of degree $-2d$ has nonzero coefficient,
so the trace cannot vanish. Suppose $d$ is odd. Then $-2d+2$ is
divisible by $4$, and the coefficient of this degree is the coefficient
of $y^{2}$ in $(1+y^{2})(u(y)u(\zeta y))^{-1}$. A direct expansion
shows that this coefficient is 
\[
u_{0}^{-2}\bigl(1+((1+\zeta)^{2}-\zeta)\alpha^{2}\bigr).
\]
For $p=1$, this is $u_{0}^{-2}(1+\iota\alpha^{2})$; its vanishing
would force $\alpha^{2}=\iota$. For $p=3$, it is $u_{0}^{-2}(1-\iota\alpha^{2})$;
its vanishing would force $\alpha^{2}=-\iota$.   Both are impossible by
Lemma~\ref{lem:iota-nonsquare}.  Thus the trace in
\eqref{eq:odd-trace-zero} is nonzero, a contradiction.  
\end{proof}

\section{Reduction to a minimal commutator}

\label{sec:reduction}

We now use the filtration associated with the $x$-adic valuation.
For $m\in\mathbb Z$, put
\[
D_{\ge m}=\{d\in D:v_x(d)\ge m\}.
\]
Its associated graded ring is
\[
G=\operatorname{gr}D
=
\bigoplus_{m\in\mathbb Z}
D_{\ge m}/D_{\ge m+1}.
\]
For standard terminology concerning filtered and associated graded
algebras, see \cite[Chapter~X, \S5]{Lang}.

Writing $X$ for the image of $x$, we identify
\[
G\simeq F[X,X^{-1};\sigma],
\qquad
Xa=\sigma(a)X.
\]
If
\[
0\ne B=b_mx^m+b_{m+1}x^{m+1}+\cdots,
\qquad b_m\ne0,
\]
we write
\[
\operatorname{in}(B)=b_mX^m.
\]

Since $\sigma$ has order $4$, put
\[
T=X^4,
\qquad
E=K[T,T^{-1}]=Z(G).
\]
For homogeneous elements one has
\[
(aX^q)(bX^p)
=
a\sigma^q(b)X^{q+p},
\]
and therefore
\begin{equation}
[aX^q,bX^p]=0
\quad\Longleftrightarrow\quad
a\sigma^q(b)=b\sigma^p(a).
\label{eq:homogeneous-commuting}
\end{equation}

Put $\tau=\sigma^2$, so
\[
L=F^\tau,\qquad F=L\oplus Ly.
\]
For brevity, we write $v=v_x$ from now on.
This section proves that if $A=[B,C]$, then one may choose $B,C$
with
\[
v(B)+v(C)=-2=v(A).
\]
\subsection*{Homogeneous rank-two and maximal types}

Let $W=aX^{q}\in G$ be homogeneous and noncentral. Here $E[W]$ denotes
the $E$-subalgebra of $G$ generated by $W$. The words ``rank two''
and ``rank four'' below refer to the rank of this commutative $E$-algebra,
not to the homogeneous degree $q$ of $W$.

We call $W$ \emph{$L$-type} if $q\equiv0\pmod 4$ and $a\in L\setminus K$.
Since $T=X^{4}$ is a central unit and $L=K(z)$, one has 
\[
E[W]=E[a]=E[z]=E\oplus Ez.
\]
Thus $E[W]$ has rank $2$ over $E$.

We call $W$ \emph{$X^{2}$-type} if $q\equiv2\pmod 4$ and $a\tau(a)\in K^{\times}$.
Then 
\[
W^{2}=a\tau(a)T^{q/2}\in E^{\times},
\]
and the two homogeneous degrees $0$ and $q$ are distinct modulo
$4$. Hence 
\[
E[W]=E\oplus EW,
\]
again of rank $2$ over $E$.

All remaining noncentral homogeneous elements generate commutative
algebras of rank $4$ over $E$; we call these \emph{maximal homogeneous
types}. There are three possibilities. If $q$ is odd, then 
\[
\begin{aligned}W^{4} & =\operatorname{N}_{F/K}(a)T^{q}\in E^{\times},\\
E[W] & =E\oplus EW\oplus EW^{2}\oplus EW^{3}.
\end{aligned}
\]
because the four displayed powers have distinct degrees modulo $4$.
If $q\equiv0\pmod 4$ and $a\notin L$, then $K(a)=F$, since $L$
is the unique quadratic subfield of the cyclic quartic extension $F/K$,
and therefore 
\[
E[W]=E[a]=F[T,T^{-1}].
\]
Finally, suppose $q\equiv2\pmod 4$ and $a\tau(a)\notin K$. Then
$a\tau(a)\in L\setminus K$, so $E[a\tau(a)]=E[z]$, and 
\[
E[W]=E[z]\oplus E[z]W.
\]
The four elements $1,z,W,zW$ are $E$-linearly independent, by their
degrees modulo $4$ and the $K$-linear independence of $1,z$. Thus
this algebra also has rank $4$ over $E$.

The maximal terminology is justified by the following direct calculation. 
\begin{lemma}
\label{lem:maximal-centralizer} If $W$ is of maximal homogeneous
type, then 
\[
C_{G}(W)=E[W].
\]
\end{lemma}

\begin{proof}
Since $W$ is homogeneous, the centralizer $C_{G}(W)$ is graded:
if a finite sum of homogeneous terms commutes with $W$, then each
of its homogeneous terms commutes with $W$. It therefore suffices
to consider a homogeneous element $V=bX^{p}$ satisfying $VW=WV$.

Suppose first that $q$ is odd. There are unique $j\in\{0,1,2,3\}$
and $k\in\mathbb{Z}$ such that $p=4k+jq$. The degree-zero element
\[
R=T^{-k}VW^{-j}\in F
\]
commutes with $W$. By \eqref{eq:homogeneous-commuting}, this means
$\sigma^{q}(R)=R$. Since $q$ is odd, the fixed field of $\sigma^{q}$
is $K$. Thus $R\in K$ and $V=RT^{k}W^{j}\in E[W]$.

Suppose next that $q\equiv0\pmod 4$ and $a\notin L$. The commuting
relation is $a=\sigma^{p}(a)$. If $p\not\equiv0\pmod 4$, the fixed
field of $\sigma^{p}$ is either $K$ or $L$, contradicting $a\notin L$.
Hence $p\equiv0\pmod 4$ and $V\in F[T,T^{-1}]=E[W]$.

Finally suppose that $q\equiv2\pmod 4$ and put $h=a\tau(a)\in L\setminus K$.
The commuting relation is 
\begin{equation}
a\tau(b)=b\sigma^{p}(a).\label{eq:maximal-x2-centralizer}
\end{equation}
If $p$ is odd, division by $ab$ and application of $\tau$ show
that 
\[
\begin{aligned}\frac{\tau(b)}{b} & =\frac{\sigma^{p}(a)}{a},\\
\frac{b}{\tau(b)} & =\frac{\sigma^{p+2}(a)}{\tau(a)}.
\end{aligned}
\]
Multiplying these identities gives $\sigma^{p}(h)=h$, impossible
because $h\in L\setminus K$ and an odd power of $\sigma$ acts nontrivially
on $L$. Thus $p$ is even. If $p\equiv0\pmod 4$, equation \eqref{eq:maximal-x2-centralizer}
gives $\tau(b)=b$, so $b\in L$. If $p\equiv2\pmod 4$, it gives
$\tau(b/a)=b/a$, so $b/a\in L$. In the first case $V\in E[z]$,
and in the second $V\in E[z]W$. Hence in either case $V\in E[W]$. 
\end{proof}
We shall also need the simultaneous centralizer in the two quadratic
configurations that survive below. 
\begin{lemma}
\label{lem:biquadratic-centralizer} Let $w=cX^{n}$ with $n\equiv2\pmod 4$
and $c\tau(c)\in K^{\times}$. Then 
\[
C_{G}(z,w)=E[z,w]=E[z]\oplus E[z]w.
\]
\end{lemma}

\begin{proof}
Again it suffices to consider a homogeneous term $V=bX^{p}$. Commuting
with $z$ forces $p$ to be even, because $\sigma^{p}(z)=z$ exactly
for even $p$. If $p\equiv0\pmod 4$, the relation $Vw=wV$ becomes
$\tau(b)=b$, so $b\in L$. If $p\equiv2\pmod 4$, it becomes $c\tau(b)=b\tau(c)$,
hence $b/c\in L$. Thus every homogeneous term in the simultaneous
centralizer lies respectively in $E[z]$ or in $E[z]w$. The reverse
inclusion is immediate, because $z$ and $w$ commute and $w^{2}=c\tau(c)T^{n/2}\in E^{\times}$. 
\end{proof}

\subsection*{A first-obstruction lemma}
\begin{lemma}
\label{lem:obstruction} Let 
\[
u=b_{0}X^{m},w=c_{0}X^{n}
\]
be commuting homogeneous elements of $G$ such that 
\[
u^{2},w^{2}\in E^{\times}.
\]
Suppose that the simultaneous centralizer 
\[
H:=C_{G}(u,w)
\]
is commutative. Suppose $N\ge1$ and $b_{1},\dots,b_{N}$, $c_{1},\dots,c_{N}$
in $F$ satisfy 
\[
R_{k}:=\sum_{i+j=k}\bigl(b_{i}\sigma^{m+i}(c_{j})-c_{j}\sigma^{n+j}(b_{i})\bigr)=0\text{ for }0\le k\le N-1.
\]
Then 
\[
R_{N}\in\bigl\{\xi\sigma^{m+N}(c_{0})-c_{0}\sigma^{n}(\xi):\xi\in F\bigr\}+\bigl\{ b_{0}\sigma^{m}(\eta)-\eta\sigma^{n+N}(b_{0}):\eta\in F\bigr\}.
\]
\end{lemma}

\begin{proof}
Put 
\[
p_{i}=b_{i}X^{m+i},q_{i}=c_{i}X^{n+i},
\]
and work in $G[\varepsilon]/(\varepsilon^{N+1})$ with 
\[
\begin{aligned}\mathcal{P} & =u+\sum_{i=1}^{N}p_{i}\varepsilon^{i},\\
\mathcal{Q} & =w+\sum_{j=1}^{N}q_{j}\varepsilon^{j}.
\end{aligned}
\]
The coefficient identities give 
\begin{equation}
\begin{aligned}{}[\mathcal{P},\mathcal{Q}] & =\omega_{N}\varepsilon^{N},\\
\omega_{N} & =R_{N}X^{m+n+N}.
\end{aligned}
\label{eq:elementary-obstruction-commutator}
\end{equation}

Set 
\[
\begin{aligned}\vartheta_{u} & =\operatorname{Ad}_{u}:G\longrightarrow G,\\
\vartheta_{u}(g) & =ugu^{-1},
\end{aligned}
\]
and similarly 
\[
\begin{aligned}\vartheta_{w} & =\operatorname{Ad}_{w}:G\longrightarrow G,\\
\vartheta_{w}(g) & =wgw^{-1}.
\end{aligned}
\]
Thus $\operatorname{Ad}_{u}$ and $\operatorname{Ad}_{w}$ denote conjugation by $u$
and $w$, respectively. Since $u^{2}$ and $w^{2}$ are central and
$u$ commutes with $w$, the maps $\vartheta_{u}$ and $\vartheta_{w}$
are commuting involutions. As $2\in E^{\times}$, they give a simultaneous
eigenspace decomposition 
\[
G=\bigoplus_{\delta_{1},\delta_{2}\in\{\pm1\}}G_{\delta_{1},\delta_{2}},
\]
where 
\[
G_{\delta_{1},\delta_{2}}=\{g\in G:\vartheta_{u}(g)=\delta_{1}g\text{ and }\vartheta_{w}(g)=\delta_{2}g\}.
\]
The corresponding projections are 
\[
\pi_{\delta_{1},\delta_{2}}=\frac{1}{4}(1+\delta_{1}\vartheta_{u})(1+\delta_{2}\vartheta_{w}).
\]
The common $+1$ eigenspace is 
\[
G_{+,+}=C_{G}(u,w)=H,
\]
and is commutative by hypothesis.

We normalize the coefficients of $\mathcal{P}$ and $\mathcal{Q}$
successively. Suppose, for some $1\le j<N$, that all coefficients
of orders below $j$ lie in $H$. Since these lower coefficients commute
with one another, the coefficient of $\varepsilon^{j}$ in $[\mathcal{P},\mathcal{Q}]=0\pmod{\varepsilon^{N}}$
is 
\begin{equation}
[p_{j},w]+[u,q_{j}]=0.\label{eq:order-j-linearized}
\end{equation}
Write $p_{j}^{\delta_{1},\delta_{2}}=\pi_{\delta_{1},\delta_{2}}(p_{j})$
and similarly for $q_{j}$. For $g\in G_{\delta_{1},\delta_{2}}$
one has 
\[
\begin{aligned}{}[g,u] & =(1-\delta_{1})gu,\\{}
[g,w] & =(1-\delta_{2})gw,\\{}
[u,g] & =(\delta_{1}-1)gu.
\end{aligned}
\]
Thus the $(\delta_{1},\delta_{2})$ component of \eqref{eq:order-j-linearized}
is 
\begin{equation}
(1-\delta_{2})p_{j}^{\delta_{1},\delta_{2}}w+(\delta_{1}-1)q_{j}^{\delta_{1},\delta_{2}}u=0.\label{eq:eigenspace-linearized}
\end{equation}
It follows that 
\[
\begin{aligned}q_{j}^{-,+} & =0,\\
p_{j}^{+,-} & =0,\\
p_{j}^{-,-}w & =q_{j}^{-,-}u.
\end{aligned}
\]
Define 
\[
\begin{aligned}\chi_{j}^{-,+} & =-\frac{1}{2}p_{j}^{-,+}u^{-1},\\
\chi_{j}^{+,-} & =-\frac{1}{2}q_{j}^{+,-}w^{-1},\\
\chi_{j}^{-,-} & =-\frac{1}{2}p_{j}^{-,-}u^{-1},\\
\chi_{j} & =\chi_{j}^{-,+}+\chi_{j}^{+,-}+\chi_{j}^{-,-}.
\end{aligned}
\]
The elements $u^{\pm1}$ and $w^{\pm1}$ belong to $G_{+,+}$, so
each displayed $\chi_{j}^{\delta_{1},\delta_{2}}$ lies in the indicated
eigenspace. The preceding relations now give 
\[
\begin{aligned}p_{j}+[\chi_{j},u] & \in H,\\
q_{j}+[\chi_{j},w] & \in H.
\end{aligned}
\]
For the $(-,-)$ component, for example, 
\[
2\chi_{j}^{-,-}w=-p_{j}^{-,-}u^{-1}w=-q_{j}^{-,-},
\]
where we used $p_{j}^{-,-}w=q_{j}^{-,-}u$ and $uw=wu$.

Work in the truncated ring $G[\varepsilon]/(\varepsilon^{N+1})$.
Since $\chi_{j}\varepsilon^{j}$ is nilpotent, $1+\chi_{j}\varepsilon^{j}$
is a unit. More precisely, 
\[
(1+\chi_{j}\varepsilon^{j})^{-1}=\sum_{k=0}^{\left\lfloor \frac{N}{j}\right\rfloor }(-1)^{k}\chi_{j}^{k}\varepsilon^{kj}.
\]
In particular, 
\[
(1+\chi_{j}\varepsilon^{j})^{-1}=1-\chi_{j}\varepsilon^{j}+O(\varepsilon^{j+1}).
\]
Conjugate both $\mathcal{P}$ and $\mathcal{Q}$ by $1+\chi_{j}\varepsilon^{j}$.
Since 
\[
\begin{aligned} & (1+\chi_{j}\varepsilon^{j})\mathcal{P}(1+\chi_{j}\varepsilon^{j})^{-1}\\
 & =u+\sum_{i<j}p_{i}\varepsilon^{i}+(p_{j}+[\chi_{j},u])\varepsilon^{j}+O(\varepsilon^{j+1}),
\end{aligned}
\]
and similarly for $\mathcal{Q}$, this conjugation leaves all coefficients
of orders below $j$ unchanged and replaces the coefficients of order
$j$ by 
\[
p_{j}+[\chi_{j},u],q_{j}+[\chi_{j},w].
\]
It also leaves the right-hand side of \eqref{eq:elementary-obstruction-commutator}
unchanged modulo $\varepsilon^{N+1}$, because $j\ge1$. Repeating
this for $j=1,\dots,N-1$, we may therefore assume that 
\[
p_{i},q_{i}\in H\text{ for }1\le i<N
\]
while \eqref{eq:elementary-obstruction-commutator} still holds with
the same $\omega_{N}$.

Since $H$ is commutative, all commutators among these lower coefficients
vanish. Taking the coefficient of $\varepsilon^{N}$ now gives 
\[
\omega_{N}=[p_{N}',w]+[u,q_{N}']
\]
for some $p_{N}',q_{N}'\in G$ (the coefficients of order $N$ after
the successive conjugations). Taking the homogeneous component of
degree $m+n+N$, only the degree-$m+N$ part of $p_{N}'$ and the
degree-$n+N$ part of $q_{N}'$ contribute. Writing these parts as
$\xi X^{m+N}$ and $\eta X^{n+N}$ gives exactly the two sets in the
statement. 
\end{proof}

\subsection*{The leading gap is impossible}
\begin{proposition}
\label{prop:gap} If $A=[B,C]$ and $(B,C)$ is chosen with $d=v(B)+v(C)$
maximal, then $d=-2$. 
\end{proposition}

\begin{proof}
By \eqref{eq:comm-val}, $d\le v(A)=-2$, and the set of attainable
values of $d$ is a nonempty subset of $\{\dots,-3,-2\}$ bounded
above, so a maximum exists. Assume $d<-2$. Write 
\[
B=\sum_{i\ge0}b_{i}x^{m+i},C=\sum_{j\ge0}c_{j}x^{n+j},m=v(B),n=v(C),
\]
and set $N=-2-d>0$, $u=\operatorname{in}(B)=b_{0}X^{m}$, $w=\operatorname{in}(C)=c_{0}X^{n}$.
For each $k\ge0$, the coefficient of $x^{d+k}$ in the commutator
$[B,C]$ is 
\[
R_{k}=\sum_{i+j=k}\bigl(b_{i}\sigma^{m+i}(c_{j})-c_{j}\sigma^{n+j}(b_{i})\bigr).
\]
Indeed, the summand $[b_{i}x^{m+i},c_{j}x^{n+j}]$ contributes to
degree $(m+i)+(n+j)=d+k$ exactly when $i+j=k$. Since $[B,C]=A$
and $A=(1+z)x^{-2}+\frac{r}{2}z$ has no terms in the degrees $d,d+1,\ldots,-3$,
the coefficients in those degrees vanish. As $N=-2-d$, this gives
\[
R_{0}=R_{1}=\cdots=R_{N-1}=0.
\]
The first nonzero degree of $A$ is $-2=d+N$, and its coefficient
there is $1+z$; hence 
\[
R_{N}=1+z.
\]
In particular $R_{0}=0$, so $[u,w]=R_{0}X^{d}=0$ in $G$.

\emph{Reduction to the rigid case.} First suppose that $u\in E[w]$.
Decompose the coefficients in $E=K[T,T^{-1}]$ into $K$-multiples
of powers of $T$. Since $u$ is homogeneous of degree $m$ and $w$
is homogeneous of degree $n$, the degree-$m$ part of such an expression
gives a finite formula 
\begin{equation}
u=\sum_{(q,j)\in I}\lambda_{q,j}T^{q}w^{j}.\label{eq:central-lift-expression}
\end{equation}
Here $\lambda_{q,j}\in K$, and $I\subset\mathbb{Z}\times\mathbb{Z}_{\ge0}$
is finite and every nonzero summand satisfies 
\[
4q+jn=m.
\]
Lift \eqref{eq:central-lift-expression} to $D$ by putting 
\[
Y=\sum_{(q,j)\in I}\lambda_{q,j}x^{4q}C^{j}.
\]
Because $x^{4}$ is central and each $C^{j}$ commutes with $C$,
we have $Y\in C_{D}(C)$. Also $v(x^{4q}C^{j})=4q+jn=m$ for each
summand, and the initial term of $Y$ is exactly the right-hand side
of \eqref{eq:central-lift-expression}, namely $u$. Hence replacing
$B$ by $B-Y$ raises $v(B)$ while preserving $[B,C]$, contradicting
maximality. Therefore $u\notin E[w]$, and by symmetry $w\notin E[u]$.

In particular neither is central, and neither is of maximal homogeneous
type. Indeed, if $u$ were maximal, Lemma~\ref{lem:maximal-centralizer}
and $[u,w]=0$ would give $w\in C_{G}(u)=E[u]$, contradicting $w\notin E[u]$;
the same argument applies to $w$. Therefore $E[u]$ and $E[w]$ are
distinct rank-two commutative $E$-algebras, each of $L$-type or
$X^{2}$-type. If both were $L$-type, then $E[u]=E[z]=E[w]$,
contradicting distinctness. After replacing $(B,C)$ by $(C,-B)$
if necessary, exactly one of the two cases occurs: 
\[
\text{(i) both }X^{2}\text{-type:}m\equiv n\equiv2\pmod 4,b_{0}=\rho c_{0},0\ne\rho\in Kz,c_{0}\tau(c_{0})\in K^{\times};
\]
\[
\text{(ii) mixed:}m\equiv0\pmod 4,n\equiv2\pmod 4,b_{0}=z,c_{0}\tau(c_{0})\in K^{\times}.
\]
(In case (i), the equality $[u,w]=0$ gives $b_{0}/c_{0}\in L$; in
case (ii), it gives $b_{0}=\tau(b_{0})$, hence $b_{0}\in L$. One
first removes the $K$-part by the same central-lift argument. In
the mixed case one then uses an inverse scalar rescaling $(B,C)\mapsto(\kappa^{-1}B,\kappa C)$,
which preserves $[B,C]$, to make the remaining $L$-type coefficient
exactly $z$.)

We now check the hypotheses on $H=E[u,w]$ explicitly. In the mixed
case, $u=zT^{m/4}$, so $E[u,w]=E[z,w]$. In the case where both terms
are of $X^{2}$-type, write $\rho=\kappa z$ with $0\ne\kappa\in K$
and $m-n=4l$. Since $u=\kappa zT^{l}w$ and $w^{-1}\in E[w]$, we
again get $E[u,w]=E[z,w]$. Thus in both cases 
\[
H=E[z,w].
\]
Moreover 
\[
w^{2}=c_{0}\tau(c_{0})T^{n/2}\in E^{\times}.
\]
In the mixed case $u=zT^{m/4}$, so $u^{2}\in E^{\times}$; in the
case where both terms are of $X^{2}$-type, $u=\kappa zT^{l}w$, so
again $u^{2}\in E^{\times}$. Finally, Lemma~\ref{lem:biquadratic-centralizer}
gives directly 
\[
C_{G}(u,w)=C_{G}(E[u,w])=C_{G}(z,w)=E[z,w]=H.
\]
In particular this simultaneous centralizer is commutative.

\emph{The contradiction.} Lemma~\ref{lem:obstruction} applies, so
$R_{N}=1+z$ lies in 
\[
\{\xi\sigma^{m+N}(c_{0})-c_{0}\sigma^{n}(\xi):\xi\in F\}+\{b_{0}\sigma^{m}(\eta)-\eta\sigma^{n+N}(b_{0}):\eta\in F\}.
\]
In case (i), $N\equiv2\pmod 4$, $m+N\equiv0$ and $n\equiv2$, so
the first set is $c_{0}\bigl(\xi-\tau(\xi)\bigr)\subseteq c_{0}Ly$;
and $b_{0}=\rho c_{0}$ with $\rho\in L$ makes the second set $b_{0}\bigl(\tau(\eta)-\eta\bigr)\subseteq c_{0}Ly$
as well. In case (ii), $N\equiv0\pmod 4$, the first set is again
$c_{0}(\xi-\tau(\xi))\subseteq c_{0}Ly$, while the second is $\{z\eta-\eta\sigma^{n+N}(z):\eta\in F\}=\{0\}$:
here $\sigma^{m}=\operatorname{id}$ since $m\equiv0\pmod 4$, and $\sigma^{n+N}(z)=\sigma^{2}(z)=z$
since $n+N\equiv2\pmod 4$ and $\sigma^{2}$ fixes $L\ni z$, so $z\eta-\eta\sigma^{n+N}(z)=z\eta-\eta z=0$.
Either way $1+z\in c_{0}Ly$.

But this is impossible. If $1+z=c_{0}\ell y$ with $\ell\in L$, applying
$\tau$ (which fixes $L$ and sends $y\mapsto-y$) gives $1+z=-\tau(c_{0})\ell y$,
and multiplying the two equalities, 
\[
\frac{(1+z)^{2}}{z}=-c_{0}\tau(c_{0})\ell^{2}\in K^{\times}L^{\times2}.
\]
This is impossible by a norm argument. Indeed, if $(1+z)^{2}/z=\kappa q^{2}$
with $\kappa\in K^{\times}$ and $q\in L^{\times}$, then $z\in K^{\times}L^{\times2}$.
Applying $\operatorname{N}_{L/K}$ gives $-s=\operatorname{N}_{L/K}(z)\in K^{\times2}$.
Since $-1=\iota^{2}$ is already a square in $K$, this would force
$s\in K^{\times2}$, contradicting the nonsquareness of $s$ proved
at the beginning of \S\ref{sec:counterexample}. This contradiction
shows $d=-2$. 
\end{proof}

\section{The minimal computation}

\label{sec:minimal}

Suppose $A=[B,C]$. By Proposition~\ref{prop:gap} we may choose
$B,C$ with 
\[
v(B)+v(C)=-2.
\]
Write $m=v(B)$, $n=v(C)$, and 
\[
\operatorname{in}(B)=b_{0}X^{m},\operatorname{in}(C)=c_{0}X^{n}.
\]
The coefficient of degree $-2$ in $[B,C]$ is 
\[
b_{0}\sigma^{m}(c_{0})-c_{0}\sigma^{n}(b_{0})=1+z.
\]
Since $m+n=-2$, the residues of $(m,n)$ modulo $4$ are one of 
\[
(0,2),(2,0),(1,1),(3,3).
\]
The two odd cases are impossible. Indeed, the case $(m,n)\equiv(1,1)\pmod 4$
gives 
\[
1+z=f_{b_{0}}^{1,1}(c_{0}),
\]
and the case $(m,n)\equiv(3,3)\pmod 4$ gives 
\[
1+z=f_{b_{0}}^{3,3}(c_{0}).
\]
Both alternatives contradict Lemma~\ref{lem:odd-obstruction}. Hence
one valuation is congruent to $0$ modulo $4$ and the other to $2$
modulo $4$. Replacing $(B,C)$ by $(C,-B)$ if necessary, and multiplying
one factor by a central power of $x^{4}$ and the other by the inverse
central power, we may therefore assume 
\[
v(B)=0,v(C)=-2.
\]
Thus 
\[
B=a+px+b_{2}x^{2}+\cdots,C=cx^{-2}+hx^{-1}+\cdots.
\]
The coefficient of $x^{-2}$ in $[B,C]$ gives 
\[
1+z=(a-\sigma^{2}(a))c.
\]

Write $a=a_{0}+\lambda y$ with $a_{0},\lambda\in L$. Then $a-\sigma^{2}(a)=2\lambda y$,
so $\lambda\ne0$ and 
\[
c=\frac{1+z}{2\lambda y}\in Ly,\sigma^{2}(c)=-c.
\]
If $c^{2}\in K$, then 
\[
\frac{(1+z)^{2}}{z}=4c^{2}\lambda^{2}\in K^{\times}L^{\times2},
\]
which is impossible by the norm argument used in \S\ref{sec:reduction};
hence $K(c^{2})=L$ and we may write $a_{0}=\mu-\nu c^{2}$ with $\mu,\nu\in K$.
Now $Y=\mu+\nu x^{4}C^{2}\in C_{D}(C)$ has initial term $\mu+\nu c\sigma^{2}(c)=\mu-\nu c^{2}=a_{0}$;
replacing $B$ by $B-Y$ leaves $[B,C]$ unchanged and removes $a_{0}$.
So we may assume 
\[
\begin{aligned}a & =\lambda y,\\
\lambda & =\alpha+\beta z,\\
c & =\gamma y,\\
\gamma & =\frac{1+z}{2\lambda z}\in L.
\end{aligned}
\]
Here $\alpha,\beta\in K$.

Write $p=P+Qy$ with $P,Q\in L$.  Since $\sigma^{-1}(y)=-\iota y$ and $\sigma^{-1}(\ell)=\overline{\ell}$
for $\ell\in L$, we have
\[
\sigma^{-1}(a)=-\iota\overline{\lambda}y.
\]
Hence
\[
a-\sigma^{-1}(a)
=
\bigl(\lambda+\iota\overline{\lambda}\bigr)y.
\]
Put
\[
d=\lambda+\iota\overline{\lambda}.
\]
Since $\lambda=\alpha+\beta z$,
\[
d=(1+\iota)\alpha+(1-\iota)\beta z.
\]
Thus $d\ne0$: otherwise the $K$-linear independence of $1$ and $z$
would give $\alpha=\beta=0$, contradicting $\lambda\ne0$.

The element $A$ has no term of degree $-1$.  Hence the coefficient
of $x^{-1}$ in $[B,C]=A$ must vanish, and therefore
\[
(a-\sigma^{-1}(a))h
+p\sigma(c)-\sigma^{2}(p)c=0.
\]
Since $a-\sigma^{-1}(a)=dy\ne0$, this equation uniquely determines
$h$:
\begin{equation}
h=-(dy)^{-1}
\bigl(p\sigma(c)-\sigma^{2}(p)c\bigr).
\label{eq:h-determined}
\end{equation}

If $k$ denotes the coefficient of $x^0$ in $C$, then $[a,k]=0$
because $a,k\in F$.  Hence the coefficient of $x^0$ gives
\[
\frac r2z
=
f_p^{1,-1}(h)+f_{b_2}^{2,-2}(c).
\]

Since $\sigma^{2}(c)=-c$ and $\sigma^{-2}=\tau$, one has 
\[
f_{b_{2}}^{2,-2}(c)=-c\bigl(b_{2}+\tau(b_{2})\bigr)\in Ly=\ker(\operatorname{Tr}_{F/L}).
\]
Applying $\operatorname{Tr}_{F/L}$ (which on $L$ is multiplication by $2$)
therefore yields the key identity 
\begin{equation}
rz=\operatorname{Tr}_{F/L}\bigl(f_{p}^{1,-1}(h)\bigr).\label{eq:rz}
\end{equation}

\begin{lemma}
\label{lem:compute} With the notation above, 
\[
\operatorname{Tr}_{F/L}\bigl(f_{p}^{1,-1}(h)\bigr)=-\frac{2}{z(\alpha^{2}+\beta^{2}s)}\bigl(\operatorname{N}_{L/K}(P)+\iota s\operatorname{N}_{L/K}(Q)\bigr).
\]
\end{lemma}

\begin{proof}
Recall that
\[
d=\lambda+\iota\overline{\lambda}\ne0,
\qquad
c=\gamma y,
\qquad
\gamma=\frac{1+z}{2\lambda z}.
\]
A direct calculation gives
\[
\operatorname{N}_{L/K}(d)
=
d\overline d
=
2\iota(\alpha^{2}+\beta^{2}s).
\]
Moreover,
\begin{equation}
\begin{aligned}
\gamma\lambda
&=\frac{1+z}{2z},\\
\overline{\gamma}\,\overline{\lambda}
&=\frac{z-1}{2z}.
\end{aligned}
\label{eq:gl}
\end{equation}

Write $h=H+Ry$ with $H,R\in L$. Using 
\[
\sigma(c)=\iota\overline{\gamma}y\qquad\text{and}\qquad\sigma^{2}(p)=P-Qy,
\]
we obtain 
\[
p\sigma(c)-\sigma^{2}(p)c=Qz(\gamma+\iota\overline{\gamma})+P(\iota\overline{\gamma}-\gamma)y.
\]
Multiplying $(dy)^{-1}=d^{-1}y/z$ into the relation 
\[
(dy)h=-(p\sigma(c)-\sigma^{2}(p)c)
\]
gives 
\[
H=-PA_{1},\qquad R=-QA_{2},
\]
where 
\[
\begin{aligned}A_{1} & =d^{-1}(\iota\overline{\gamma}-\gamma),\\
A_{2} & =d^{-1}(\gamma+\iota\overline{\gamma}).
\end{aligned}
\]

Now 
\[
\sigma(h)=\overline{H}+\iota\overline{R}y\qquad\text{and}\qquad\sigma^{-1}(p)=\overline{P}-\iota\overline{Q}y.
\]
Thus the $L$-part of 
\[
f_{p}^{1,-1}(h)=p\sigma(h)-\sigma^{-1}(p)h
\]
is 
\[
P\overline{H}-\overline{P}H+\iota z(Q\overline{R}+\overline{Q}R).
\]
Since $\operatorname{Tr}_{F/L}(U+Vy)=2U$ for all $U,V\in L$, and $H=-PA_{1}$,
$R=-QA_{2}$, we obtain 
\begin{equation}
\operatorname{Tr}_{F/L}\bigl(f_{p}^{1,-1}(h)\bigr)=2\operatorname{N}_{L/K}(P)(A_{1}-\overline{A_{1}})-2\iota z\operatorname{N}_{L/K}(Q)(A_{2}+\overline{A_{2}}).\label{eq:half}
\end{equation}

It remains to evaluate the two combinations. Since 
\[
\overline{d}=\overline{\lambda}+\iota\lambda,
\]
we have 
\[
\begin{aligned}\iota\overline{d}+d & =2\iota\overline{\lambda},\\
\iota d+\overline{d} & =2\iota\lambda.
\end{aligned}
\]
Using 
\[
d^{-1}=\frac{\overline{d}}{\operatorname{N}_{L/K}(d)}\qquad\text{and}\qquad\overline{d}^{-1}=\frac{d}{\operatorname{N}_{L/K}(d)},
\]
we obtain 
\begin{align*}
A_{1}-\overline{A_{1}} & =\frac{1}{\operatorname{N}_{L/K}(d)}\bigl[\overline{\gamma}(\iota\overline{d}+d)-\gamma(\iota d+\overline{d})\bigr]\\
 & =\frac{2\iota}{\operatorname{N}_{L/K}(d)}\bigl(\overline{\gamma}\overline{\lambda}-\gamma\lambda\bigr),\\
A_{2}+\overline{A_{2}} & =\frac{1}{\operatorname{N}_{L/K}(d)}\bigl[\gamma(\iota d+\overline{d})+\overline{\gamma}(\iota\overline{d}+d)\bigr]\\
 & =\frac{2\iota}{\operatorname{N}_{L/K}(d)}\bigl(\gamma\lambda+\overline{\gamma}\overline{\lambda}\bigr).
\end{align*}
By \eqref{eq:gl}, 
\[
\overline{\gamma}\overline{\lambda}-\gamma\lambda=-\frac{1}{z},\qquad\gamma\lambda+\overline{\gamma}\overline{\lambda}=1,
\]
and 
\[
\operatorname{N}_{L/K}(d)=2\iota(\alpha^{2}+\beta^{2}s).
\]
Therefore 
\[
\begin{aligned}A_{1}-\overline{A_{1}} & =-\frac{1}{z(\alpha^{2}+\beta^{2}s)},\\
A_{2}+\overline{A_{2}} & =\frac{1}{\alpha^{2}+\beta^{2}s}.
\end{aligned}
\]
Substituting these identities into \eqref{eq:half} and using $z^{2}=s$,
we obtain 
\begin{align*}
\operatorname{Tr}_{F/L}\bigl(f_{p}^{1,-1}(h)\bigr) & =-\frac{2}{\alpha^{2}+\beta^{2}s}\left(\frac{\operatorname{N}_{L/K}(P)}{z}+\iota z\operatorname{N}_{L/K}(Q)\right)\\
 & =-\frac{2}{z(\alpha^{2}+\beta^{2}s)}\bigl(\operatorname{N}_{L/K}(P)+\iota s\operatorname{N}_{L/K}(Q)\bigr).
\end{align*}
\end{proof}
Combining \eqref{eq:rz} with Lemma~\ref{lem:compute} gives
\begin{equation}
rz
=
-\frac{2}{z(\alpha^{2}+\beta^{2}s)}
\bigl(
\operatorname{N}_{L/K}(P)
+\iota s\operatorname{N}_{L/K}(Q)
\bigr).
\label{eq:parity-identity}
\end{equation}

\section{The parity contradiction}

\label{sec:parity}

We use diagonal notation $\langle a_{1},\dots,a_{m}\rangle$ for $\sum a_{i}X_{i}^{2}$.
Write $P=P_{0}+P_{1}z$ and $Q=Q_{0}+Q_{1}z$, where $P_{i},Q_{i}\in K$.
Then 
\[
\begin{aligned}\operatorname{N}_{L/K}(P)+\iota s\operatorname{N}_{L/K}(Q) & =P_{0}^{2}-sP_{1}^{2}+\iota sQ_{0}^{2}-\iota s^{2}Q_{1}^{2}\\
 & =\varphi(P_{0},P_{1},Q_{0},Q_{1}),
\end{aligned}
\]
where 
\[
\varphi=\langle1,-s,\iota s,-\iota s^{2}\rangle.
\]
Moreover, $\alpha^{2}+\beta^{2}s$ is a value of $\langle1,s\rangle$.
We call a quadratic form \emph{anisotropic} if it has no nontrivial
zero. We use Springer's theorem in the form given by \cite[Chapter~VI, Proposition~1.9]{Lam}.
Thus, over a nondyadic complete discretely valued field with uniformizer
$\pi$, a diagonal form with unit coefficients is anisotropic if and
only if its reduction is anisotropic, and a form $q_{0}\perp\pi q_{1}$
with unit coefficients is anisotropic if and only if the reductions
of $q_{0}$ and $q_{1}$ are both anisotropic. We shall also use the
following immediate consequence: if a diagonal form with unit coefficients
has anisotropic reduction, then every nonzero value that it represents
has even valuation. Indeed, after factoring the smallest power $\pi^{e}$
from the coordinates of a representing vector, the remaining primitive
vector has nonzero value modulo $\pi$, so the represented value has
valuation $2e$.

Apply this first over $K_{0}=\mathbb{Q}(\iota)((s))$ with uniformizer
$s$. For $\langle1,s\rangle=\langle1\rangle\perp s\langle1\rangle$,
the two residue forms are $\langle1\rangle$ over $\mathbb{Q}(\iota)$
and hence are anisotropic. Thus $\langle1,s\rangle$ is anisotropic
over $K_{0}$. For $\varphi=\langle1,-s,\iota s,-\iota s^{2}\rangle$,
replacing $-\iota s^{2}$ by $-\iota$ (a square multiple) gives 
\[
\varphi\cong\langle1,-\iota\rangle\perp s\langle-1,\iota\rangle.
\]
The residue forms $\langle1,-\iota\rangle$ and
$\langle-1,\iota\rangle$ are anisotropic over $\mathbb{Q}(\iota)$:
a nontrivial zero of either form would imply that $\iota$ is a
square in $\mathbb{Q}(\iota)$, contrary to
Lemma~\ref{lem:iota-nonsquare}.
Hence $\varphi$ is anisotropic over $K_{0}$.

Now pass to $K=K_{0}((r))$ with uniformizer $r$. The forms $\langle1,s\rangle$
and $\varphi$ have coefficients in $K_{0}$, hence unit coefficients
for the $r$-adic valuation, and their reductions are the same anisotropic
forms over $K_{0}$. By the preceding consequence of \cite[Chapter~VI, Proposition~1.9]{Lam},
every nonzero value of either form over $K$ has even $r$-adic valuation.
We use the same notation $v_{r}$ for the unramified extension of
this valuation to $L=K(z)=\mathbb{Q}(\iota)((z))((r))$; in particular
$v_{r}(z)=0$.

In \eqref{eq:parity-identity}, the scalar $-\tfrac{2}{z(\alpha^{2}+\beta^{2}s)}$
has even $r$-valuation ($v_{r}(2)=v_{r}(z)=0$ and $\alpha^{2}+\beta^{2}s$
is a value of $\langle1,s\rangle$), and the bracket $\operatorname{N}_{L/K}(P)+\iota s\operatorname{N}_{L/K}(Q)$
is a value of $\varphi$. Hence the right-hand side, if nonzero, has
even $r$-adic valuation; and it cannot vanish, since the left-hand
side $rz$ is nonzero. But $rz$ has $r$-adic valuation $1$. This
contradiction is the proof. 
\begin{theorem}
\label{thm:main} Let $K=\mathbb{Q}(\iota)((s))((r))$, $F=K(y)$
with $y^{4}=s$ and $\sigma(y)=\iota y$. Then 
\[
A=(1+y^{2})x^{-2}+\frac{r}{2}y^{2}
\]
has reduced trace zero in the degree-four central division algebra
$D=F((x;\sigma))$, but is not a single commutator. 
\end{theorem}

\begin{proof}
Reduced trace zero was checked in \S\ref{sec:counterexample}. If
$A=[B,C]$, then \S\ref{sec:reduction} reduces to a minimal pair,
\S\ref{sec:minimal} yields \eqref{eq:parity-identity}, and \S\ref{sec:parity}
shows its two sides have $r$-adic valuations of opposite parity.
Hence no such $B,C$ exist. 
\end{proof}

\section*{Declaration of generative AI and AI-assisted technologies in the
manuscript preparation process}

During the preparation of this work, the author used ChatGPT-5 and Claude Opus as generative-AI tools for exploratory, verification, and editorial purposes. The author independently obtained the observations presented in Section 2 and, from them, conjectured the existence of counterexamples. ChatGPT-5 was subsequently used to assist in the search for explicit candidate counterexamples. The reduction argument in Section 4 was developed by the author, with ChatGPT-5 used to assist with calculations and to check intermediate steps. During the preparation of the final manuscript, ChatGPT-5 was also used for language polishing and grammatical correction. Claude Opus was used only for independent checking of the arguments and manuscript, including the detection of notation collisions and other consistency issues. The author reviewed and independently verified all mathematical arguments and all AI-assisted material, edited the resulting text as necessary, and takes full responsibility for the content of the article.






\end{document}